\documentclass[11pt,reqno]{amsart} 
\usepackage[utf8]{inputenc}
\usepackage{amsfonts,amsmath,amsthm,amssymb}
\usepackage[justification=centering]{caption}
\usepackage{color}
\usepackage{fullpage}
\usepackage{enumitem}   
\usepackage{hyperref,float}
\usepackage{mathtools}
\usepackage{xcolor}
\usepackage{thm-restate}
\usepackage{xurl}
\usepackage{thmtools}
\usepackage{hyperref}
\usepackage{cleveref}

\usepackage{enumitem}   
\usepackage{mathtools}
\usepackage{tikz}
\usepackage{thmtools}
\usepackage{thm-restate}
\newcommand{\asc}{\mathrm{asc}}
\newcommand{\Asc}{\mathrm{Asc}}
\newcommand{\des}{\mathrm{des}}
\newcommand{\Des}{\mathrm{Des}}
\newcommand{\fr}[1]{\text{FR}_{#1}}
\newcommand{\Fub}{\mathrm{Fub}}
\newcommand{\N}{\mathbb{N}}
\newcommand{\Out}{\mathcal{O}}
\newcommand{\pf}[1]{\mathrm{PF}_{#1}}
\newcommand{\Sym}{\mathfrak{S}}
\newcommand{\ufr}[1]{\mathrm{UFR}_{#1}}
\newcommand{\upf}[1]{\mathrm{UPF}_{#1}}
\declaretheorem[name=Theorem,numberwithin=section]{rethm}

\newtheorem{proposition}[rethm]{Proposition}
\newtheorem{lemma}[rethm]{Lemma}

\newtheorem{corollary}[rethm]{Corollary}
\newtheorem{remark}[rethm]{Remark}

\newtheorem{definition}[rethm]{Definition}
\newtheorem{example}[rethm]{Example}

\title{
    Parking functions, Fubini rankings, and Boolean intervals in the weak order of $\mathfrak{S}_n$
}

\author[Elder]{Jennifer Elder}
\address[J.~Elder]{Department of Computer Science, Mathematics and Physics, Missouri Western State University}
\email{\textcolor{blue}{\href{mailto:jelder8@missouriwestern.edu}{jelder8@missouriwestern.edu}}}

\author[Harris]{Pamela E. Harris}
\address[P.~E.~Harris]{Department of Mathematical Sciences, University of Wisconsin-Milwaukee}
\email{\textcolor{blue}{\href{mailto:peharris@uwm.edu}{peharris@uwm.edu}}}

\author[Kretschmann]{Jan Kretschmann}
\address[J.~Kretschmann]{Department of Mathematical Sciences, University of Wisconsin-Milwaukee}
\email{\textcolor{blue}{\href{mailto:kretsc23@uwm.edu}{kretsc23@uwm.edu}}}

\author[Mart\'inez Mori]{J. Carlos Mart\'inez Mori}
\address[J.~C. Mart\'inez Mori]{H. Milton Stewart School of Industrial and Systems Engineering, Georgia Institute of Technology}
\email{\textcolor{blue}{\href{mailto:jcmm@gatech.edu}{jcmm@gatech.edu}}}

\begin{document}

\subjclass{Primary: 05A05; Secondary: 06A07, 05A15, 05A19}
\keywords{permutations, weak order, Fubini rankings, parking functions, Fibonacci numbers}

\begin{abstract}
Let $\mathfrak{S}_n$ denote the symmetric group and let $W(\mathfrak{S}_n)$ denote the weak order of $\mathfrak{S}_n$.
Through a surprising connection to a subset of parking functions, which we call \textit{unit Fubini rankings}, we provide a complete characterization and enumeration for the total number of Boolean intervals in $W(\mathfrak{S}_n)$ and the total number of Boolean intervals of rank $k$ in $W(\mathfrak{S}_n)$. 
Furthermore, for any $\pi\in\mathfrak{S}_n$, we establish that the number of Boolean intervals in $W(\mathfrak{S}_n)$ with minimal element $\pi$ is a product of Fibonacci numbers. 
We conclude with some directions for further study.
\end{abstract}

\maketitle

\section{Introduction}

A poset is called \emph{Boolean} if it is isomorphic to the poset of subsets of a set $I$ ordered by inclusion.
The term \emph{Boolean poset} is inherited from \textit{Boolean algebras}, given that one of the most familiar examples of a Boolean algebra is the power set $2^{I}$. 
If $|I| = k < \infty$, then a Boolean poset is a distributive lattice, making it a ranked poset.
Henceforth, we let $B_k$ denote a Boolean poset of rank $k$.

Boolean posets appear frequently in combinatorics, especially as intervals (subposets) within larger structures.
In these cases, they are referred to as \emph{Boolean intervals}.
One notable example is that of Boolean intervals in the \textit{right weak (Bruhat) order} on the symmetric group $\Sym_n$~\cite{bjorner2009shape, tenner2006reduced,tenner2007pattern,tenner2022interval}. 
The weak order of $\Sym_n$, denoted $W(\Sym_n)$, is generated by transpositions \[s_i = (i,i+1)\] for $i \in [n-1]$, where $[n] \coloneqq \{1, 2, \ldots, n\}$.
That is, cover relations arise from the (right-hand side) application of a single simple transposition, and simple transpositions generate this group.
Figure~\ref{fig: s6} highlights a $B_3$ interval in $W(\Sym_6)$.

\begin{figure}[ht]
    \centering{
    \includegraphics[trim={0.75cm 0.5cm 0.75cm 0.5cm},clip,width=0.8\linewidth]{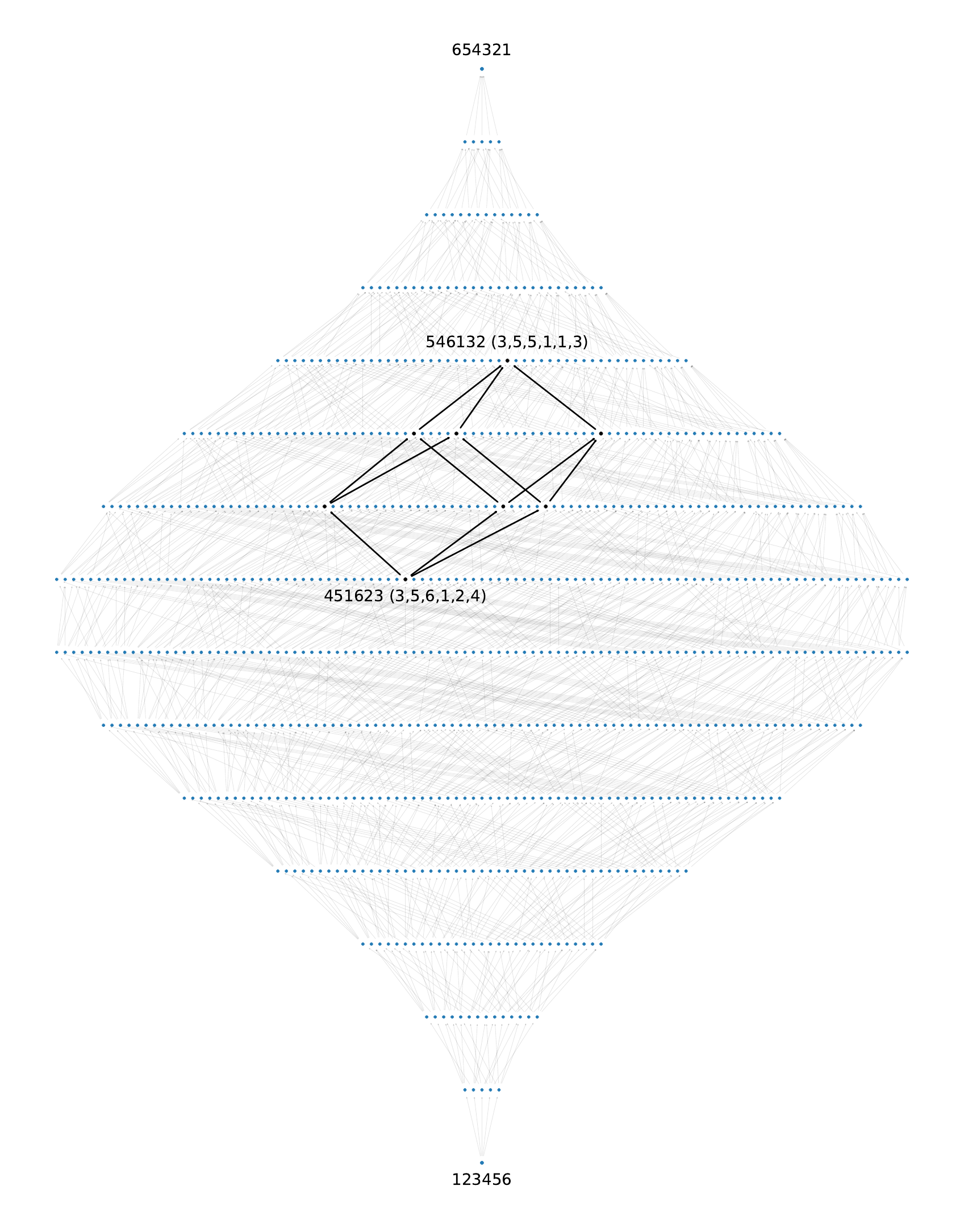}
    }
    \caption{
    Illustration of $W(\Sym_6)$.
    A Boolean interval $B_3$ with minimal element $451623$ and maximal element $546132$, written in one-line notation, is highlighted. 
    The tuples $(3, 5, 6, 1, 2, 4)$ and $(3, 5, 5, 1, 1, 3)$ indicate the unit Fubini rankings associated with the minimal and maximal Boolean subintervals of $B_3$ with minimal element $451623$ and maximal element $546132$, respectively.
    }
    \label{fig: s6}
\end{figure}

Tenner established that Boolean posets appear as intervals $[v,w]$ in the weak order if and only if $v^{-1}w$ is a permutation composed of only commuting generators \cite[Corollary 4.4]{tenner2022interval}. 
We recall that generators $s_i$ and $s_j$ commute whenever $|i-j|>1$.
We provide more background on the weak order and Boolean intervals in Section~\ref{sec: background}.
Tenner also established that Boolean intervals with a generator as minimal element are enumerated by products of at most two Fibonacci numbers~\cite[Proposition~5.9]{tenner2022interval}. 
Our first result generalizes Tenner's result as follows.

\begin{restatable}{rethm}{sizeOfFiberThm}
\label{thm: size of fiber}
Let $\pi = \pi_1 \pi_2 \cdots \pi_n \in \Sym_{n}$ be in one-line notation and partition its ascent set $\Asc(\pi)=\{i \in [n-1]: \pi_i<\pi_{i+1}\}$ into maximal blocks $b_1, b_2, \ldots, b_k$ of consecutive entries. 
Then, the number of Boolean intervals $[\pi, w]$ in $W(\Sym_n)$ with fixed minimal element $\pi$ and arbitrary maximal element $w$ (including the case $\pi=w$) is given by
\begin{align*}
    \prod_{i=1}^kF_{|b_i|+2},
\end{align*}
where $F_{\ell}$ is the $\ell$th Fibonacci number, and $F_1=F_2 = 1$.
\end{restatable}

Our proofs of Theorem~\ref{thm: size of fiber} and subsequent results rely on a class of combinatorial objects we refer to as \emph{unit Fubini rankings}, which are a subset of Fubini rankings. 
A tuple $\alpha \in [n]^n$ is a \emph{Fubini ranking} of length $n$ if it records a valid ranking over $n$ competitors with ties allowed, where the distinct values in the tuple are the \emph{ranks} (refer to Definition~\ref{def:Fubini ranking} for a technical definition).
For example, $(4,1,1,1)$ is a Fubini ranking since competitors $2$, $3$, and $4$ are tied and rank first, and competitor $1$ ranks fourth.
However, $(1,1,2,3)$ is not a Fubini ranking since competitors $1$ and $2$ are tied and rank first, implying no competitor can rank second (indeed, the next available rank would be third).
Fubini rankings are enumerated by the Fubini numbers (OEIS \href{https://oeis.org/A000670}{A000670}), which first appeared in the work of Cayley in enumerating trees with certain properties \cite{cayley1875analytical}. 
For more on competitions allowing ties we recommend Mendelson~\cite{mendelson1982races}.

\emph{Unit} Fubini rankings are the subset of Fubini rankings in which ranks are shared by at most two competitors. 
For example, $(4,2,2,1)$ is a unit Fubini ranking, whereas $(4,1,1,1)$ is not a unit Fubini ranking.
The term ``unit'' is derived from \emph{unit interval parking functions}, these are parking functions in which cars park at most one spot away from their preference \cite{hadaway2022combinatorial}. 
We employ this language as in Corollary~\ref{lemma: ufr intersection} we show that the set of unit Fubini rankings is precisely the intersection between the set of Fubini rankings and the set of unit interval parking functions.

Prior to stating our main result, we acknowledge the double use of the word ``rank'' when describing the rank of a poset and the ranks used in a Fubini ranking.

\begin{restatable}{rethm}{nkThm}
\label{thm:nkThm}
The set of unit Fubini rankings with $n-k$ distinct ranks is in bijection with the set of Boolean intervals in $W(\Sym_n)$ of rank $k$.
\end{restatable}

We leverage the parking interpretation of unit Fubini rankings to count unit Fubini rankings with $n-k$ distinct ranks.
In turn, by Theorem~\ref{thm:nkThm}, this provides a count of Boolean intervals in $W(\Sym_n)$ of rank $k$.
\begin{restatable}{rethm}{newFormulaThm}
\label{newFormulaThm}
Let $f(n,k)$ denote the number of Boolean intervals in $W(\Sym_n)$ of rank $k$.
Then,
\begin{equation}
\label{eq: f(n,k)}
    f(n,k) = \frac{n!}{2^k}\binom{n-k}{k}.
\end{equation}
\end{restatable}
Equation~\eqref{eq: f(n,k)} recovers the following known results:
\begin{itemize}
    \item $f(n,0)=n!$ is the number of permutations $(B_0)$ in $\Sym_n$ (OEIS \href{https://oeis.org/A000142}{A000142}), 
    \item $f(n,1)=\frac{n!(n-1)}{2}$ is the number of edges ($B_1$) in $W(\Sym_n)$ (OEIS \href{https://oeis.org/A001286}{A001286}), and 
    \item $f(n,2)=\frac{n!(n-2)(n-3)}{8}$ is the number of 4-cycles ($B_2$) in $W(\Sym_n)$ (OEIS \href{https://oeis.org/A317487}{A317487}).
\end{itemize}  
To the best of our knowledge, we are the first to establish a general formula for $f(n,k)$.

By setting $q=1$ into the exponential generating function \cite[Exercise~3.185(h)]{stanley2012enumerative}
\begin{equation}
    F(x,q) 
    = \sum_{n\geq 0} \sum_{k\geq 0} f(n,k)q^k\frac{x^n}{n!} 
    = \frac{1}{1-x-\frac{q}{2}x^2},
\end{equation}
Stanley~\cite{stanley2016mathoverflow} points out that the \emph{total} number of Boolean intervals in $W(\Sym_n)$ (OEIS \href{https://oeis.org/A080599}{A080599}) satisfies the recurrence relation  
\begin{equation}
    f(n+1) = (n+1)f(n)+\binom{n+1}{2}f(n-1),
\end{equation}
where $f(0)=1$ and $f(1)=1$. 
However, Stanley did not provide a closed formula for the values of $f(n,k)$ as we do in Equation~\eqref{eq: f(n,k)}.

The remainder of this paper is organized as follows. 
In Section~\ref{sec: background} we provide necessary background on Boolean intervals, parking functions, and Fubini rankings.
In Section~\ref{sec:ufr} we present preliminary results on unit Fubini rankings, including an inequality characterization and operations that preserve unit Fubini rankings. 
In Section~\ref{sec: bijection} we prove Theorem~\ref{thm:nkThm}, establishing a bijection between unit Fubini rankings with $n-k$ distinct ranks and Boolean intervals in $W(\Sym_n)$ of rank $k$.
In Section~\ref{sec: enumerations}, we prove Theorem~\ref{newFormulaThm}, giving a closed formula for the number of Boolean intervals in $W(\Sym_n)$ of rank $k$.
We conclude with Section \ref{sec:future}, providing directions for future study.

\section{Background}
\label{sec: background}

We now present necessary background on Boolean intervals in $W(\Sym_n)$. 
We then provide some history and results related to Fubini rankings and their interpretation as parking functions, which we will use in the proofs of our main results.

\subsection{Boolean Intervals in the Weak Order}

Boolean posets are constructed by subsets of a set $I$ ordered by inclusion. 
Figure~\ref{fig: ranks0to3} illustrates some small examples.
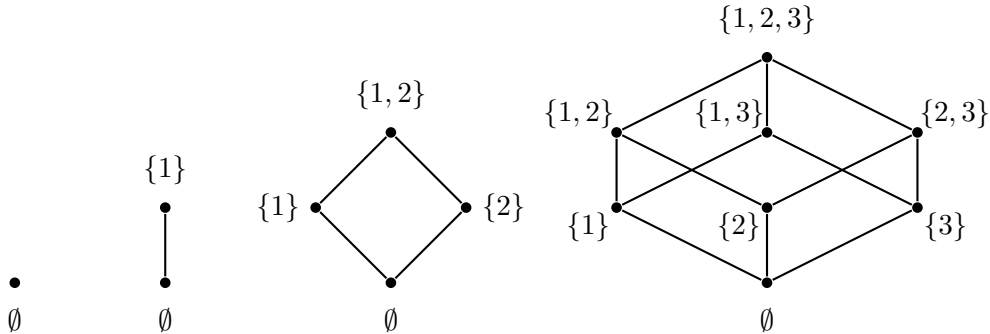
\begin{figure}[ht]
    \centering
    \begin{tikzpicture}
    [inner sep=0pt,thick, dot/.style={fill=black,circle,minimum size=4pt}]

\node[dot](a6) at (0,0) {};
\node at (0,-.5) {$\emptyset$};

\node[dot](a4) at (2,0) {};
\node at (2,-.5) {$\emptyset$};
\node[dot](a5) at (2,1) {};
\node at (2,1.5) {$\{1\}$};
\draw (a4)--(a5);
    
\node[dot](a0) at (5,0) {};
\node at (5,-.5) {$\emptyset$};
\node[dot](a1) at (6,1) {};
\node at (6.5,1) {$\{ 2\}$};
\node[dot](a2) at (4,1) {};
\node at (3.5,1) {$\{1\}$};
\node[dot](a3) at (5,2) {};
\node at (5,2.5) {$\{ 1,2\}$};
\draw (a0)--(a1)--(a3)--(a2)--(a0);

\node[dot](a7) at (10,0) {};
\node at (10,-.5) {$\emptyset$};

\node[dot](a8) at (8,1) {};
\node at (7.625,0.75) {$\{1\}$};
\node[dot](a9) at (10,1) {};
\node at (9.625,0.75) {$\{2\}$};
\node[dot](a10) at (12,1) {};
\node at (12.375,0.75) {$\{3\}$};

\node[dot](a11) at (8,2) {};
\node at (7.5,2.25) {$\{1,2\}$};
\node[dot](a12) at (10,2) {};
\node at (9.5,2.25) {$\{1,3\}$};
\node[dot](a13) at (12,2) {};
\node at (12.5,2.25) {$\{2,3\}$};

\node[dot](a14) at (10, 3) {};
\node at (10,3.5) {$\{1,2,3\}$};

\draw (a7)--(a8);
\draw (a7)--(a9);
\draw (a7)--(a10);
\draw (a8)--(a11);
\draw (a8)--(a12);
\draw (a9)--(a11);
\draw (a9)--(a13);
\draw (a10)--(a12);
\draw (a10)--(a13);
\draw(a11)--(a14);
\draw(a12)--(a14);
\draw(a13)--(a14);

\end{tikzpicture}
    \caption{Boolean posets $B_0$, $B_1$, $B_2$, and $B_3$.}
    \label{fig: ranks0to3}
\end{figure}

Unless specified, we write permutations in one-line notation. The following definition plays a key role in our proof of Theorem~\ref{thm: size of fiber}.
\begin{definition}\label{def:descentsandascents}
For a permutation $\sigma=\sigma_1\sigma_2\cdots\sigma_n\in \Sym_n$, the \emph{ascent set of $\sigma$} is given by
\begin{equation*}
    \Asc(\sigma)=\{j\in[n-1]\;:\; \sigma_j<\sigma_{j+1}\}.
\end{equation*}
Let $\asc(\sigma)=|\Asc(\sigma)|$ denote the number of ascents of $\sigma$.
Similarly, the \emph{descent set of $\sigma$} is given by
\begin{equation*}
    \Des(\sigma)=\{j\in[n-1]\;:\; \sigma_j>\sigma_{j+1}\}.
\end{equation*}
Let $\des(\sigma)=|\Des(\sigma)|$ denote the number of descents of $\sigma$.
\end{definition}

The \textit{right weak (Bruhat) order}, denoted $W(\mathfrak{S}_n)$, is a partial order on $\mathfrak{S}_n$. 
Its cover relations are defined by the application of a single simple (adjacent) transposition on the right-hand side. 
That is, $\tau \lessdot \sigma$ if and only if $\tau s_i =\sigma$ for some $i\in \Des(\sigma)$. 
In general, if $\tau \leq \sigma$, then there exists a sequence $s_{i_1}, \ldots, s_{i_k}$ of simple transpositions such that $\tau s_{i_1}\ldots s_{i_k} = \sigma$.

In fact, $W(\mathfrak{S}_n)$ is a bounded lattice for all $n\geq 2$ \cite{stanley2012enumerative}.
In one-line notation, its minimal element is $12\cdots n$ while its maximal element is $n(n-1) \cdots 21$. 
Figure~\ref{fig: s_4} illustrates $W(\Sym_4)$ with its elements written in one-line notation.

\begin{figure}[ht]
    \centering
    \includegraphics[width=0.5\linewidth]{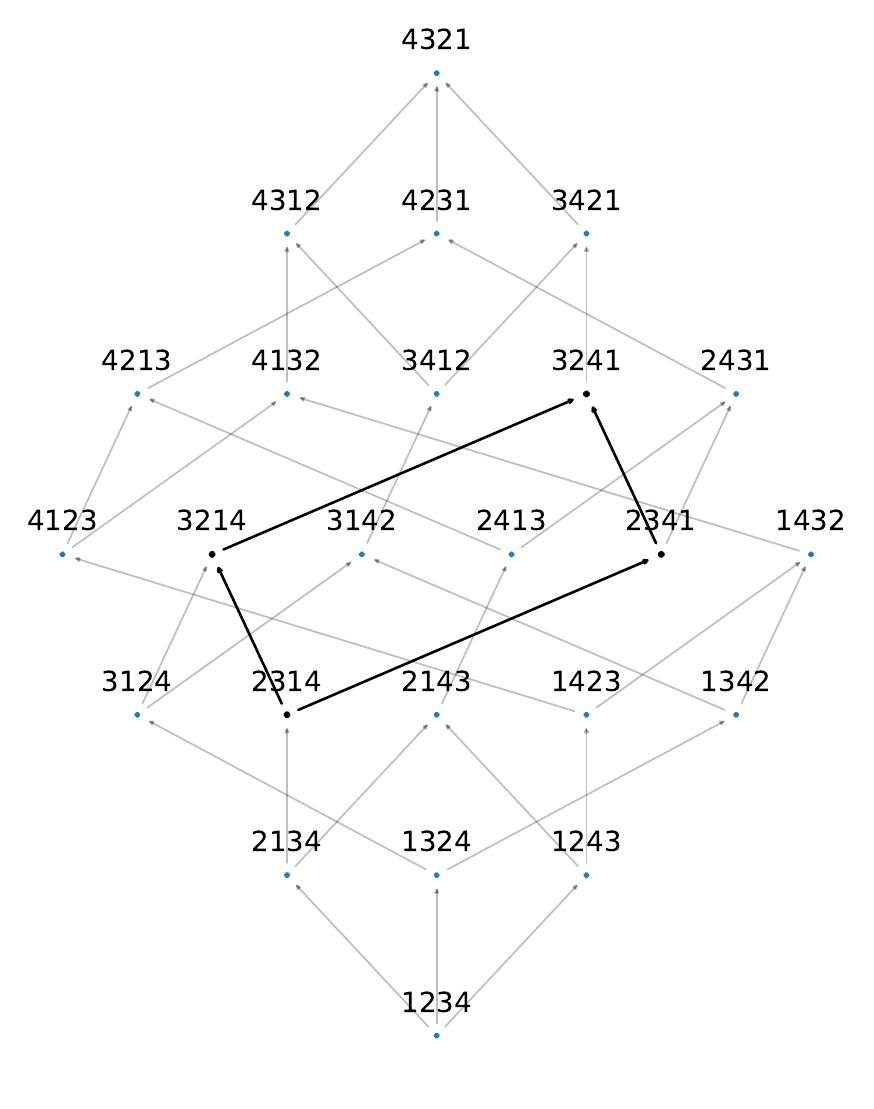}
  \caption{
  Illustration of $W(\mathfrak{S}_4)$ with a highlighted  Boolean interval $B_2$.}
  \label{fig: s_4}
\end{figure}

\begin{remark}
In a similar way, we can define the \emph{weak left (Bruhat) order}, where $\tau \leq \sigma$ if and only if there exists a sequence $s_{k_1}, \ldots, s_{k_m}$ of simple transpositions such that $\sigma = s_{k_m} \cdots s_{k_1} \tau$. 
The two weak orders are distinct, but isomorphic under the map $\sigma \mapsto \sigma^{-1}$.
\end{remark}

A subset $[\sigma, \tau] \subseteq W(\Sym_{n})$ is an (nonempty) interval
if $\sigma \leq \tau$ and $\pi \in [\sigma, \tau]$ whenever $\sigma \leq \pi \leq \tau$.
Tenner established that Boolean intervals in $W(\Sym_n)$ have the structure $[v,w]$ if and only if $v^{-1}w$ is a permutation composed of only commuting generators \cite[Corollary 4.4]{tenner2022interval}. 

\begin{example}
In Figure~\ref{fig: s_4}, if $\pi\in \Sym_4$, then the interval $[\pi,\pi]$ is a Boolean interval of rank zero. 
In addition, all intervals $[\pi, \pi s_i]$ where $i\in \Asc (\pi)$ are Boolean intervals of rank one. 
Finally, if $\Asc(\pi) = \{1,3\}$, then the interval $[\pi , \pi s_1s_3]$ is a Boolean interval of rank two. 
For example, the interval $[2314, 3241]$, highlighted in Figure~\ref{fig: s_4}, is one of the six Boolean intervals of rank two in $W(\Sym_4)$.
\end{example}

\subsection{Parking Functions, Unit Interval Parking Functions, and Fubini Rankings}

A tuple $\alpha=(a_1,a_2,\ldots,a_n)\in[n]^n$ is a \emph{parking function} of length $n$ if its weakly increasing rearrangement $\alpha'=(a_1',a_2',\ldots,a_n')$ satisfies $a_i'\leq i$ for all $i\in[n]$. 
For example $\alpha=(1,6,4,4,3,3,2)$ is a parking function of length seven as its weakly increasing rearrangement $\alpha'=(1,2,3,3,4,4,6)$ satisfies the inequality conditions.
However, $\alpha=(1,5,4,6,6,3,7)$ is not a parking function, as its weakly increasing rearrangement 
$\alpha'=(1,3,4,5,6,6,7)$ does not satisfy the inequality condition for $i=2$.
Let $\pf{n}$ denote the set of parking functions of length $n$.
Parking functions were introduced by Konheim and Weiss~\cite{konheim1966occupancy}, who established that $|\pf{n}|=(n+1)^{n-1}$ for all $n \geq 1$.

One can interpret a tuple $\alpha=(a_1,a_2,\ldots,a_n)\in[n]^n$ as encoding the parking preferences of $n$ cars that attempt to park, one at a time, on a one-way street with $n$ parking spots. 
When car $i \in [n]$ arrives, it attempts to park in its preferred spot $a_i$.
If spot $a_i$ is unoccupied, car $i$ parks there.
Otherwise, car $i$ continues driving down the one-way street until it parks in the first unoccupied spot, if there is one.
If no such spot exists, then car $i$ is unable to park.
If all cars are able to park, then $\alpha$ is a parking function.
Figure~\ref{fig: parking} illustrates the order in which cars park on the street when $\alpha=(1,6,4,4,3,3,2)$. 
We refer to the resulting parking order as the \emph{outcome} of $\alpha$. 

\begin{figure}[ht]
    \centering
    \begin{tikzpicture}
\node at (3.5,-1) {$\alpha=(1,6,4,4,3,3,2)$};
    \draw[step=1cm,gray,very thin] (0,0) grid (7,1);
    \draw[ultra thick,->] (-.75,.5) -- (-.25,.5);
    \draw[fill=gray!50] (.1,0.1) rectangle (.9,.9);
    \node at (.5,.5) {$1$};
    \draw[fill=gray!50] (1.1,0.1) rectangle (1.9,.9);
    \node at (1.5,.5) {$7$};        \draw[fill=gray!50] (2.1,0.1) rectangle (2.9,.9);
    \node at (2.5,.5) {$5$};
    \draw[fill=gray!50] (3.1,0.1) rectangle (3.9,.9);
    \node at (3.5,.5) {$3$};
    \draw[fill=gray!50] (4.1,0.1) rectangle (4.9,.9);
    \node at (4.5,.5) {$4$};
    \draw[fill=gray!50] (5.1,0.1) rectangle (5.9,.9);
    \node at (5.5,.5) {$2$};
    \draw[fill=gray!50] (6.1,0.1) rectangle (6.9,.9);
    \node at (6.5,.5) {$6$};
    \node at (.5,-.25) {$1$};
    \node at (1.5,-.25) {$2$};
    \node at (2.5,-.25) {$3$};
    \node at (3.5,-.25) {$4$};
    \node at (4.5,-.25) {$5$};
    \node at (5.5,-.25) {$6$};
    \node at (6.5,-.25) {$7$};
    \end{tikzpicture}
    \caption{The parking outcome of the preference tuple $\alpha=(1,6,4,4,3,3,2)$.
The upper row of data (in boxes) gives the car labels and the bottom row of data gives the parking spots.}
    \label{fig: parking}
\end{figure}
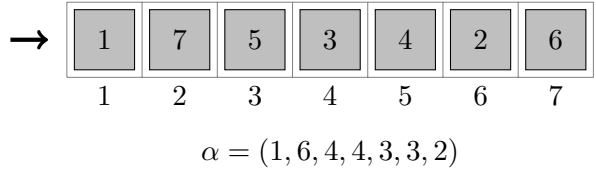

Hadaway and Harris introduced unit interval parking functions, which are the subset of parking functions in which cars park exactly at their preferred spot or one spot away~\cite{hadaway2022combinatorial}. 
For example, $(1, 2, 3, 4, 5)$, $(1, 1, 3, 4, 5)$, $(1, 1, 2, 4, 5)$ are unit interval parking functions (of length $5$), whereas $(1, 1, 1, 1, 1)$ is a parking function but not a unit interval parking function.
Let $\upf{n}$ denote the set of unit interval parking functions of length $n$.
Hadaway and Harris established bijectively that the number of unit interval parking functions of length $n$ is the Fubini numbers, also known as the ordered Bell numbers (OEIS \href{https://oeis.org/A000670}{A000670}).
That is,
\begin{align}\label{eq:fubini numbers}
 |\upf{n}|
& =
 \Fub_{n}
 =\sum_{k=1}^n k!\, S(n,k),
\end{align}
where $S(n,k)$ are Stirling numbers of the second kind (OEIS \href{https://oeis.org/A008277}{A008277}), which count the number of set partitions of $[n]$ with $k$ non-empty parts. 

To establish their result, Hadaway and Harris proved that the set of unit interval parking functions is in bijection with the set of Fubini rankings.
\begin{definition}\label{def:Fubini ranking}
  A Fubini ranking of length $n$ is a tuple $r=(r_1,r_2,\ldots,r_n)\in [n]^n$ that records a valid ranking over $n$ competitors with ties allowed (i.e., multiple competitors can be tied and have the same rank).
However, if $k$ competitors are tied and rank $i$th, then the $k-1$ subsequent ranks $i + 1, i + 2, \ldots, i + k - 1$ are disallowed.  
\end{definition}
For example, if two competitors are tied and rank first, then the second rank is disallowed and the next available rank is the third.

Similarly, $(1, 1, 3, 3, 5)$, $(1, 2, 3, 4, 5)$, $(1, 1, 1, 1, 1)$, $(3, 1, 5, 1, 3)$ are all Fubini rankings (of length $5$) while $(3, 1, 5, 1, 2)$ is not, as competitors $2$ and $4$ are tied and rank first, implying no competitor can rank second.
Let $\fr{n}$ denote the set of Fubini rankings of length $n$.

By the definition of Fubini ranking, any rearrangement of a Fubini ranking is itself a Fubini ranking.
In other words, Fubini rankings are invariant under permutations.
As we reference this fact in a later section, we state it formally below.
\begin{lemma}
\label{fubini perm invariant}
Fubini rankings are invariant under permutations.
\end{lemma}

In the remainder of this paper, we consider the intersection of Fubini rankings and unit interval parking functions, which we describe in the next section.

\section{Unit Fubini Rankings}\label{sec:ufr}

The intersection $\fr{n}\cap \upf{n}$ is non-trivial for all $n>1$ and this set plays a key role in our arguments.
\begin{definition}
    For all $n\geq 1$, the set $\ufr{n}\coloneqq\fr{n}\cap \upf{n}$ and its elements are referred to as \emph{unit Fubini rankings}. 
\end{definition}
Table~\ref{tab: UFR_n} gives the cardinality of $\ufr{n}$ for small values of $n$, agreeing with OEIS \href{https://oeis.org/A080599}{A080599}, which Stanley identifies as the number of Boolean intervals in $W(\mathfrak{S}_n)$. 
His remark motivates this work.

\begin{table}[!ht]
    \centering
    \begin{tabular}{c|c|c|c|c|c|c|c}
      $n$   &  1 & 2 & 3 & 4 & 5 & 6 & 7\\
      \hline
       $|\ufr{n}|$   & 1 & 3 & 12 & 66 & 450 & 3690 & 35280
    \end{tabular}
    \caption{The number of unit Fubini rankings with $1\leq n\leq 7$ competitors.}
\label{tab: UFR_n}
\end{table}

Bradt, Elder, Harris, Rojas Kirby, Reutercrona, Wang, and Whidden \cite{bradt2024unit}, gave a complete characterization of unit interval parking functions, which we utilize to give a characterization of unit Fubini rankings. 

\begin{definition}[\cite{bradt2024unit}]
\label{def: block structure}
Let $\alpha = (a_1, a_2, \ldots, a_n) \in \upf{n}$ and $\alpha' = (\alpha_1', \alpha_2', \ldots, \alpha_n')$ be its weakly increasing rearrangement.
Let $i_1, i_2, \ldots, i_m \in [n]$ be the increasing sequence of indices satisfying $\alpha'_{i_j} = i_j$.
The partition of $\alpha'$ as the concatenation $b_1|b_2|\cdots|b_m$ where $b_j = (\alpha_{i_j}', \alpha_{i_j + 1}',\ldots, \alpha_{i_{j+1}-1}')$ is called the \textit{block structure} of $\alpha$. 
Each part $b_j$ for $j \in [m]$ is called a \emph{block} of $\alpha$.
\end{definition}

Next we state the characterization of unit parking functions by Bradt et al. \cite[Theorem 2.9]{bradt2024unit}.

\begin{restatable}{rethm}{upf_rearrangment}$($\cite{bradt2024unit}$)$\label{thm:upf_rearrangement}
Given $\alpha = (a_1, \ldots, a_n) \in \upf{n}$, let $\alpha'$ be its weakly increasing rearrangement and $\alpha'=\pi_1\,|\,\pi_2\,|\,\dots\,|\,\pi_m$ be the block structure of $\alpha$.
\begin{enumerate}
    \item  There are
\begin{align}
\binom{n}{|\pi_1|,\ldots, |\pi_m|}   \label{eq:rearrange count}
\end{align}
possible rearrangements $\sigma$ of $\alpha$ such that $\sigma$ is still a unit interval parking function.
\item A rearrangement $\sigma$ of $\alpha$ is in $\upf{n}$ if and only if the entries in $\sigma$ respect the relative order of the entries in each of the blocks $\pi_1,\pi_2,\ldots,\pi_m$.
\end{enumerate}
\end{restatable}

For our purposes, we only need the following result, which follows from Theorem \ref{thm:upf_rearrangement}. 

\begin{corollary}
\label{cor: block internals}
Let $\alpha \in \upf{n}$ and $b_1\,|\,b_2\,|\,\cdots\,|\,b_m$ be its block structure.
For each $j \in [m]$, let $i_j$ be the minimal element of $b_j$.
Consider any $j \in [m - 1]$.
If $|b_j| = 1$, then $b_j = (i_j)$ and $i_{j+1} = i_j + 1$.
Otherwise, if $|b_j| = 2$, then $b_j = (i_j, i_j)$ and $i_{j+1} = i_j + 2$.
Otherwise, $|b_j| \geq 3$, $b_j = (i_j, i_j, \underbrace{i_j + 1, i_j + 2, \ldots, i_j + |b_j| - 2}_{|b_j| - 2 \text{ terms }})$ and $i_{j+1} = i_j + |b_j|$.
\end{corollary}

We now give a characterization of unit Fubini rankings based on their block structure.
We employ this technical result in our proof of Theorem~\ref{thm:nkThm}.

\begin{restatable}{rethm}{intersectionThm}
\label{thm: intersection}
Let $\alpha \in \upf{n}$ and $b_1 \,|\, b_2 \,|\, \cdots \,|\, b_m$ be its block structure.
Then, $\alpha \in \ufr{n}$ if and only if $|b_j| \leq 2$ for each $j \in [m]$.
\end{restatable}
\begin{proof}
First, suppose $|b_j| \leq 2$ for each $j \in [m]$.
We need to show that $\alpha \in \ufr{n}$.
To do this, it suffices to show that for each pair $b_j, b_{j+1}$ of consecutive blocks with $j \in [m-1]$, the existence of competitors whose ranks correspond to the block $b_j$ does not disallow there being a competitor whose rank is the minimal element of block $b_{j+1}$.
Consider any such pair $b_j, b_{j+1}$ of consecutive blocks and let $i_j$ and $i_{j+1}$ be the minimal elements of blocks $b_j$ and $b_{j+1}$, respectively.
If $|b_j| = 1$, then by Corollary~\ref{cor: block internals} we know that $b_j = (i_j)$ and $i_{j+1} = i_j + 1$, so there being a competitor whose rank is $i_j$ does not disallow there being a competitor whose rank is $i_{j+1} = i_j + 1$.
If $|b_j| = 2$, then by Corollary~\ref{cor: block internals} we know that $b_j = (i_j, i_j)$ and $i_{j+1} = i_j + 2$, so there being two competitors whose ranks are both $i_j$ does not disallow there being a competitor whose rank is $i_{j+1} = i_j + 2$.

Now, suppose $|b_j| = k > 2$ for some $j \in [m]$.
We need to show that $\alpha \notin \ufr{n}$.
Let $i_j$ be the minimal element of block $b_j$ so that, by Corollary~\ref{cor: block internals}, $b_j = (i_j, i_j, i_j + 1, \ldots, i_j + k)$.
In $b_j$, $i_j$ appears twice while $i_j + 1$ appears once.
Therefore, similarly in $\alpha$, $i_j$ appears twice while $i_j + 1$ appears once.
This implies that $\alpha \notin \ufr{n}$, since there being two competitors whose ranks are both $i_j$ disallows the subsequent rank $i_j + 1$, which some competitor supposedly holds.
\end{proof}
The following corollary gives an alternative way to state Theorem \ref{thm: intersection}.
\begin{corollary}
\label{lemma: ufr intersection}
Let $\alpha\in\upf{n}$. Then $\alpha\in\ufr{n}$ if and only if $\alpha$ is a Fubini ranking with the additional constraint that ranks are shared by at most two competitors. 
\end{corollary}
We also give an inequality description of unit Fubini rankings.
\begin{corollary}
\label{cor: characterization}
Let $\alpha=(a_1,a_2,\ldots,a_n)\in[n]^n$ and $\alpha'=(a_1',a_2',\ldots,a_n')$ be its weakly increasing rearrangement.
Then, $\alpha \in \ufr{n}$ if and only if $c_i\leq a_i'\leq i$ for each $i\in[n]$, where 
\[c_i=\begin{cases}
1,&\mbox{if $i=1$}\\
i,&\mbox{if $a_{i-1}'=i-2$ and $2\leq i\leq n$}\\
i-1,&\mbox{otherwise}.
\end{cases}\]
\end{corollary}
\begin{proof}
First, let $\alpha \in \ufr{n}$. 
Then, by Theorem~\ref{thm: intersection}, the block structure $b_1 \,|\, b_2 \,|\, \cdots \,|\, b_m$ of $\alpha$ satisfies $|b_j| \leq 2$ for each $j \in [m]$.
This implies that $c_i \leq a_i' \leq i$ for each $i \in [n]$.   

Now, let $\alpha \in [n]^n$ such that $c_i \leq a_i' \leq i$ for all $i \in [n]$. 
This implies that each number $i \in [n]$ occurs at most twice in $\alpha$.
Moreover, if $i \in [n]$ occurs twice, then the next smallest number, if any, is $i+2$. 
This implies that the block structure $b_1 \,|\, b_2 \,|\, \cdots \,|\, b_m$ of $\alpha$ satisfies $|b_j| \leq 2$ for each $j \in [m]$.
By Theorem~\ref{thm: intersection}, this implies $\alpha \in \ufr{n}$. 
\end{proof}

We now take a quick aside to provide a connection between unit Fubini rankings and the Fibonacci numbers (OEIS \href{https://oeis.org/A000045}{A000045}), defined by $F_{n+1}=F_{n}+F_{n-1}$ for $n\geq 2$ and $F_1=F_2=1$.
\begin{restatable}{rethm}{fubFibThm}
\label{thm:fubFib}
Let $\ufr{n}^{\uparrow}$ be the set of weakly increasing unit Fubini rankings of length $n$. 
Then, for $n \geq 1$ we have
\begin{equation*}
    |\ufr{n}^{\uparrow}|=F_{n+1},
\end{equation*}
where $F_{n+1}$ is the $(n+1)$th Fibonacci number.
\end{restatable}
\begin{proof}
We will show that $|\ufr{n}^{\uparrow}|=|\ufr{n-1}^{\uparrow}|+|\ufr{n-2}^{\uparrow}|$, $|\ufr{2}^{\uparrow}| = 2$, and $|\ufr{1}^{\uparrow}|= 1$.
By Theorem~\ref{thm: intersection}, the block structure of any unit Fubini ranking has blocks of size at most two. 
Moreover, for any $n \in \N$, each $\alpha \in \ufr{n}$ satisfies $|\{i \in [n]: a_i = n\}| \leq 1$.
That is, among $n$ competitors, at most one is ranked $n$.
Therefore, to compute $|\ufr{n}^{\uparrow}|$, we need only consider forming a block of size two in which 2 participants tie and rank $n-1$ to any $\beta \in \ufr{n-2}^{\uparrow}$, or appending a block of size one with rank $n$ to any $\gamma \in \ufr{n-1}^{\uparrow}$.
These cases are disjoint and exhaustive, and therefore give the required recursion relation.
To conclude, note $|\ufr{1}^{\uparrow}|=|\{(1)\}|=1$ and 
$|\ufr{2}^{\uparrow}|=|\{(1,1),(1,2)\}|=2$.
\end{proof}

Lastly, we describe a set of functions on unit Fubini rankings used in future sections to establish Theorem~\ref{thm:nkThm}. 
\begin{definition}
\label{def: delta}
For each $i\in[n-1]$ define $\delta_i: \ufr{n}\to\ufr{n}$  by
\begin{equation}
\label{eq: delta}
    \delta_i(\alpha) =
    \begin{cases}
        \alpha, & \text{if } \substack{|\{j : a_j = i-1\}| = 2 \\ \text{ or } |\{j : a_j = i\}| = 2 \\ \text{ or } |\{j : a_j = i+1\}| = 2} \\
        \widehat{\alpha}(i), & \text{otherwise;}
    \end{cases}
\end{equation}
where $\widehat{\alpha}(i)$ is obtained from $\alpha$ by decreasing the singular occurrence of $i+1$ to $i$.
\end{definition}
For example, if $\alpha=(1,3,5,3,6,1,7)$, then $\delta_i(\alpha) = \alpha$, for $1\leq i\leq 4$, while
\begin{itemize}
    \item $\delta_5(\alpha)=\widehat{\alpha}(5)=(1,3,5,3,5,1,7)$, because $6$ occurs exactly once in $\alpha$ and
    \item $\delta_6(\alpha)=\widehat{\alpha}(6)=(1,3,5,3,6,1,6)$, because $7$ occurs exactly once in $\alpha$.
\end{itemize}
One can readily confirm that all of the tuples above are in $\ufr{7}$.
This illustrates the next result.
\begin{lemma}
The functions $\delta_i$ for $i \in [n-1]$ are well-defined.
\end{lemma}

\begin{proof}
Let $\alpha \in \ufr{n}$ and let $b_1 \,|\, b_2 \,|\, \cdots \,|\,b_m$ be its block structure.
Consider any fixed but arbitrary $i\in[n-1]$.
We need to show that $\delta_i(\alpha) \in \ufr{n}$.
There are two possibilities.

\noindent\textbf{Case 1}: 
Suppose $\delta_i(\alpha) = \alpha$.
The claim holds since $\alpha \in \ufr{n}$, by assumption.

\noindent\textbf{Case 2}: 
Suppose $\delta_i(\alpha) = \widehat{\alpha}(i)$. 
By definition of $\delta_i$ this means that if $i-1$, $i$, or $i+1$
appear in $\alpha$, then they appear at most once.
In addition, by Corollary~\ref{cor: characterization}, if $i+2\leq n$, then $i+2$ appears at least once in $\alpha$. 
The only change that $\delta_i$ makes to obtain $\widehat{\alpha}(i)$ from $\alpha$ occurs at the value $i+1$, which is decreased to $i$; all other entries of $\alpha$ remain unchanged.
Therefore, the only change that $\delta_i$ makes to the block structure $b_1 \,|\, b_2 \,|\, \cdots \,|\,b_m$ 
is that 
the adjacent singleton blocks containing $(i)$ and $(i+1)$
are turned into a single block of size $2$ containing $(i,i)$.
Then, Corollary~\ref{cor: block internals} guarantees that $\widehat{\alpha}(i)\in\upf{n}$ while, in turn, Theorem~\ref{thm: intersection} guarantees that $\widehat{\alpha}(i)\in\ufr{n}$, as claimed.
\end{proof}

Next we show that the functions of Definition~\ref{def: delta} commute whenever their domain is restricted to the set of permutations and they are applied on nonconsecutive indices.
\begin{restatable}{rethm}{lem: commute}
\label{lem: commute}
Let  $i , j \in [n-1]$ be nonconsecutive. If $\pi \in \Sym_n$, then $\delta_i(\delta_j(\pi)) = \delta_j(\delta_i(\pi))$.
\end{restatable}
\begin{proof}
Fix nonconsecutive integers $i, j \in [n-1]$. 
Without loss of generality, let $i<j$.
By Lemma~\ref{fubini perm invariant}, it suffices to consider only the identity permutation $\pi = 12\cdots n$ and the block structure of $\pi$ is $b_1 \,|\, b_2 \,|\, \cdots \,|\, b_n$ with singleton blocks $b_i = (i)$ for each $i \in [n]$.

If $\delta_i(\pi)$ has the block structure $1\,|\,2\,|\,\cdots\,|\,i-1\,|\,i\,i\,|\,i+2\,|\,\cdots\,|\,n-1\,|\,n$, 
then, since $i < j$, the block structure of $\delta_j(\delta_i(\pi))$ is
\[
1\,|\,2\,|\,\cdots\,|\,i-1\,|\,i\,i\,|\,i+2\,|\,\cdots\,|\,j-1\,|\,j\,j\,|\,j+2\,|\,\cdots\,|\,n-1\,|\,n.
\] 
If $i+2=j$, then the block structure would be
\[
1\,|\,2\,|\,\cdots\,|\,i-1\,|\,i\,i\,|\,j\,j\,|\,j+2\,|\,\cdots\,|\,n-1\,|\,n.
\]
On the other hand, 
$\delta_j(\pi)$ has the block structure \[1\,|\,2\,|\,\cdots\,|\,j-1\,|\,j\,j\,|\,j+2\,|\,\cdots\,|\,n-1\,|\,n.\] 
Then, since $i < j$, the block structure of  $\delta_i(\delta_j(\pi))$ is \[1\,|\,2\,|\,\cdots\,|\,i-1\,|\,i\,i\,|\,i+2\,|\,\cdots\,|\,j-1\,|\,j\,j\,|\,j+2\,|\,\cdots\,|\,n-1\,|\,n.\] 
Again, if $i+2=j$, then the block structure would be
\[1\,|\,2\,|\,\cdots\,|\,i-1\,|\,i\,i\,|\,j\,j\,|\,j+2\,|\,\cdots\,|\,n-1\,|\,n.\]
Therefore, for $\pi = 12\cdots n$, we have $\delta_i(\delta_j(\pi)) = \delta_j(\delta_i(\pi))$. 
Finally, for any $\pi \neq 12\cdots n$, the blocks $(i,i)$ and $(j,j)$ will be in the positions where the consecutive blocks $\cdots\,|\,i\,|\,i+1\,|\,\cdots$ and $\cdots\,|\,j\,|\,j+1\,|\,\cdots$ originally appeared, respectively.
\end{proof}

\begin{remark}
\label{rem:nocommute}
    In Theorem~\ref{lem: commute}, it is important that $i$ and $j$ are nonconsecutive. 
    To see this, let $\pi \in \Sym_n$ and $j=i+1$. 
    Then, the block structure of $\pi$ changes in the following way upon application of $\delta_{i+1}$ followed by $\delta_i$:
    \begin{align}\label{i}
    \delta_i(\delta_{i+1}(\pi))=\delta_i(\cdots\,|\,i-1\,|\,j\,j\,|\,i+2\,|\,\cdots) 
    = \cdots\,|\,i-1\,|\,i+1\,\,\,\,\,i+1\,|\,i+2\,|\,\cdots.
    \end{align}
    On the other hand, 
    the block structure of $\pi$ changes in the following way upon application of $\delta_{i}$ followed by $\delta_{i+1}$:
    \begin{align}
        \delta_{i+1}(\delta_i(\pi))=\delta_{i+1}(\cdots\,|\,i-1\,|\,i\,i\,|\,i+2\,|\,\cdots)=\cdots\,|\,i-1\,|\,i\,i\,|\,i+2\,|\,\cdots.\label{ii} 
    \end{align}
\end{remark}

We now generalize Definition~\ref{def: delta} to be indexed by subsets consisting of nonconsecutive integers.

\begin{definition}
    Let $I=\{i_1,i_2,\ldots,i_k\} \subset [n-1]$ be a set of pairwise nonconsecutive integers satisfying $i_1<i_2<\cdots<i_k$. 
    If $\pi \in \Sym_n$, then we define the composition 
\begin{align}
\label{eq: composition}
\delta_I(\pi)\coloneqq\delta_{i_1}\circ \delta_{i_2}\circ\cdots\circ\delta_{i_k}(\alpha).
\end{align} 
If $I=\emptyset$, then $\delta_I = \mbox{Id}$ is the identity map on $\Sym_n$.
\end{definition}

It follows from Theorem \ref{lem: commute} that the composition defined in Equation~\eqref{eq: composition} can be done in any order.

\begin{corollary}\label{lem:composition}
    Let $I=\{i_1,i_2,\ldots,i_k\}\subseteq[n-1]$ 
be a set of nonconsecutive integers. If $\pi\in \Sym_n$, 
then the composition $    \delta_I(\pi)\in \ufr{n}$.
\end{corollary}

\section{Bijection}
\label{sec: bijection}

By definition,  $\ufr{n}\subseteq \upf{n}$, hence, we can treat unit Fubini rankings as parking functions.
We define the outcome map $\Out: \ufr{n} \rightarrow \Sym_n$ by $\Out(\alpha)=\pi=\pi_1\pi_2\cdots\pi_n$
where $\pi\in\Sym_n$ is written in one-line notation and denotes the order in which the cars park on the street. 
That is, if $j\in[n]$, then $\pi_j=i$ denotes that car $i$ is the $j$th car parked on the street.  
Given $\pi\in \Sym_n$, we define the fiber of the outcome map:
\[\Out^{-1}(\pi)=
\{\alpha\in\ufr{n}: \Out(\alpha)=\pi\}.
\]

\begin{remark}\label{rem:OutWellDef}
    Since no car can park in more than one spot, $\Out$ is a well-defined map.
\end{remark}

In what follows, we write both Fubini rankings and permutations in one-line notation.
We now provide some initial technical results. 
The first result is analogous to \cite[Theorem 2.2]{harris2023outcome} and hence we omit its proof.
\begin{lemma}\label{lem:inverse_perm}
    Let $\pi\in \Sym_n$.
    Then $\alpha = \pi^{-1}$ is the unique permutation with outcome $\pi$.
\end{lemma}

For a fixed $\pi\in\Sym_n$, we are interested in determining the elements of $\Out^{-1}(\pi)$. 
Next, we provide the connection between the elements in $\Out^{-1}(\pi)$ and the ascent set $\Asc(\pi)$.

\begin{lemma}
    Let $\pi=\pi_1\pi_2\cdots\pi_n\in\Sym_n$.
    If $j\in\Asc(\pi)$, $\pi_{j+1}=i$, and $\alpha=(a_1,a_2,\ldots,a_n)\in\Out^{-1}(\pi)$, then $a_{i}\in\{j,j+1\}$.
\end{lemma}
\begin{proof}
Assume $j\in\Asc(\pi)$, so $\pi_j<\pi_{j+1}$. 
This means that car $\pi_{j+1}=i$ arrived after car $\pi_j$ and is parked immediately to the right of $\pi_{j}$. Under unit interval parking rule, there are only two ways in which car $i$ can park in spot $j+1$, either spot $j+1$ was its preference and that spot was available, or its preference was the spot $j$, which it found occupied by car $\pi_j$. 
Thus $a_i\in \{j,j+1\}$ as desired.
\end{proof}

\begin{lemma}\label{lem:ascent in pi}
Let $\pi=\pi_1\pi_2\cdots\pi_n\in\Sym_n$.
    If $\beta\in\Out^{-1}(\pi)$ has a block consisting of a repeated value $ii$, then $\pi_i<\pi_{i+1}$ and the outcome permutation $\pi$ has an ascent at index $i$.
\end{lemma}

\begin{proof}
    Suppose that in $\beta$, $b_x=i$ and $b_y=i$, where $x<y$. Car $x$ parks first, and parks in spot $i$. This means that $\pi_i=x$. Car $y$ attempts to park in spot $i$, but must park in spot $i+1$ instead. That is, $\pi_{i+1}=y$. Since $x<y$, there is an ascent in $\pi$ at $i$.
\end{proof}

\begin{proposition}\label{prop:fiberascents}
Let $\pi=\pi_1\pi_2\cdots\pi_n\in\Sym_n$ and $\alpha=\pi^{-1}\in\Out^{-1}(\pi)$. 
    Then \[\Out^{-1}(\pi)=\{\delta_I(\alpha): I\subseteq \Asc(\pi) \mbox{ with nonconsecutive entries}\}.\]
\end{proposition}

Before we prove Proposition \ref{prop:fiberascents}, we illustrate the effect of $\delta_I$ on a permutation $\pi$, when $I$ is a subset of nonconsecutive elements from $\Asc(\pi)$.
\begin{example}
    Fix $\pi = 412356$. Then  $\Asc(\pi) = \{2,3,4,5 \}$. Then $\alpha=\pi^{-1} = 234156$ is the unique permutation in $\Out^{-1}(\pi)$. 
    The only possible subsets of $\Asc(\pi)=\{2,3,4,5\}$ consisting of nonconsecutive integers are: $\emptyset$, $\{2\}$, $\{3\}$, $\{4\}$, $\{5\}$, $\{2,4\}$, $\{2,5\}$, and $\{3,5\}$ and 
    \begin{center}
    \begin{tabular}{ccc}
        $\delta_{\emptyset}(\alpha)=234156$,&
        $\delta_{\{2\}}(\alpha)=224156 $, &
        $\delta_{\{3\}}(\alpha)= 233156$, \\
        $\delta_{\{4\}}(\alpha)=234146$, &
        $\delta_{\{5\}}(\alpha)=234155$, &
        $\delta_{\{2,4\}}(\alpha)=224146$, \\
        $\delta_{\{2,5\}}(\alpha)=224155$, &
        $\delta_{\{3,5\}}(\alpha)=233155$. & 
        \end{tabular}
    \end{center}
        Straightforward computations establish that the results are unit Fubini rankings with outcome~$\pi$, we illustrate one such computation next.
        To check that any of the above expressions are  unit Fubini rankings, requires that we confirm they are a Fubini ranking, and that any rank appears at most twice. Consider $\delta_{3,5}(\alpha)=233155$, which encodes the information that competitor 4 ranks first, then competitor 1 ranks second, competitors 2 and 3 tie for rank three (disallowing rank four), and competitors 5 and 6 tie and rank fifth (disallowing rank six). Moreover, the ranks that appear, appear at most twice. This confirms that $\delta_{\{3,5\}}(\alpha)$ is a unit Fubini ranking. To conclude, we can park the cars according to $\delta_{\{3,5\}}(\alpha)$: car 1 parks in spot 2 (i.e. $\pi_2=1$), car 2 parks in spot 3 (i.e. $\pi_3=2$), car 3 parks in spot 4 (i.e. $\pi_4=3$), car 4 parks in spot 1 (i.e. $\pi_1=4$), car 5 parks in spot 5 (i.e. $\pi_5=5$), and car 6 parks in spot 6 (i.e. $\pi_6=6$). Thus, $\delta_{\{3,5\}}(\alpha)=233155$ has outcome $\pi_1\pi_2\pi_3\pi_4\pi_5\pi_6=412356=\pi$. Similar computations show all of the above listed expressions are in fact in $\Out^{-1}(\pi)$.

Given that the only subsets of $\Asc(\pi)$ consisting of nonconsecutive integers are those we have listed above, we have now confirmed that $\delta_I(\alpha)\in \Out^{-1}(\pi)$ for any such subset $I$. 
\end{example}

\begin{proof}[Proof of Proposition \ref{prop:fiberascents}]
It suffices to show
\begin{enumerate}
    \item $\Out^{-1}(\pi)\subseteq \{\delta_I(\pi^{-1}): I\subseteq \Asc(\pi) \mbox{ with nonconsecutive entries}\}$ and \label{setinclusion1}
    \item $ \{\delta_I(\pi^{-1}): I\subseteq \Asc(\pi) \mbox{ with nonconsecutive entries}\} \subseteq \Out^{-1}(\pi)$.\label{setinclusion2}
\end{enumerate}

\noindent For \eqref{setinclusion1}:  Fix $\beta \in \Out^{-1}(\pi)$ and consider the block structure of $\beta$. We will induct on the number of blocks of size two.
If the block structure of $\beta$ has no blocks of size 2, then $\beta$ is a permutation. By Lemma \ref{lem:inverse_perm}, we know that $\beta=\pi^{-1}$. If $I=\emptyset\subseteq\Asc(\pi)$,  then $\delta_{\emptyset}(\pi^{-1})=\pi^{-1}=\beta$. Thus $\beta\in\{\delta_I(\pi^{-1}): I\subseteq \Asc(\pi) \mbox{ with nonconsecutive entries}\}$, as desired. 
Now suppose that 
that the block structure of $\beta$ contains exactly one block of size two. 
Let the entries of that block be $ii$.
By Lemma \ref{lem:ascent in pi}, if $i$ appears twice in $\beta$, then $i\in\Asc(\pi)$. We must also have that $\delta_i(\pi^{-1}) = \beta$. 
Let $I=\{ i\}\subseteq\Asc(\pi)$ consisting of nonconsecutive integers. Then $\delta_{I}(\pi^{-1})=\beta$ since by definition $\delta_{\{i\}}$ when applied to $\pi^{-1}$ takes the unique value $i+1$ and replaces it with $i$. Thus $\delta_{\{i\}}(\pi^{-1})=\beta$ since we have exactly one block of size two with repeated entries $ii$.
Therefore, $\beta \in \{\delta_I(\pi^{-1}): I\subseteq \Asc(\pi) \mbox{ with nonconsecutive entries}\}$, as desired. 

Inductively, for any $\beta \in \Out^{-1}(\pi)$ with $k$ blocks of size 2, we can reconstruct the set $I$ by looking at the entries in those $k$ blocks. The indices in $I$ must all be more than one unit away, are determined by the minimum element in each block of size two, and must have all come from the ascent set of~$\pi$. Thus $\delta_I(\pi^{-1}) = \beta$, which means that $\beta \in \{\delta_I(\pi^{-1}): I\subseteq \Asc(\pi) \mbox{ with nonconsecutive entries}$ in $I=\{ i\}\}$.

\noindent For \eqref{setinclusion2}:
Let $I=\{i_1,i_2,\ldots,i_k\}\subseteq \Asc(\pi)$ consist of nonconsecutive integers. 
Without loss of generality assume $i_1<i_2<\cdots<i_k$. 
By Corollary \ref{lem:composition} we know $\delta_I(\pi^{-1})\in \ufr{n}$,
and the block structure of $\delta_I(\pi^{-1})$ is as follows:
\begin{itemize}[leftmargin=.2in]
    \item For each $i\in I$, there is a block of size two containing both instances of $i$ in $\delta_I(\pi^{-1})$, and 
    \item for each  $i\notin I$, there is a block of size one containing the only instance $i$ in $\delta_I(\pi^{-1})$.
\end{itemize}
Since the entries in $I$ are nonconsecutive, the block structure of $\delta_I(\pi^{-1})$ ensures that 
if $i\notin I$, car $\pi_i$ with preference $i$ parks in spot $i$, as needed to have outcome $\pi$. Moreover, if $i\in I$, then under $\delta_{I}(\pi^{-1})$, car $\pi_i$ has preference $i$ and parks in spot $i$, and car $\pi_{i+1}$ has preference $i$ and as $\pi_i<\pi_{i+1}$ it finds spot $i$ occupied and parks in spot $i+1$, as needed to have outcome $\pi$. Thus establishing that $\Out(\delta_I(\pi^{-1}))=\pi$, as desired.
\end{proof}

Tenner established that Boolean intervals in the weak order all have the form $[v,w]$ where $w= v \prod_{i\in I} s_i$ for some $I\subseteq \Asc(v) $ whose elements are nonconsecutive \cite[Corollary 4.4]{tenner2022interval}. 
We use this result in the proof of our main result, which we restate for convenience.

\nkThm*

\begin{proof}
Fix $\pi\in \Sym_{n}$.
Let $\mathcal{B}_n$ be the set of all Boolean intervals in $W(\Sym_n)$, and $\mathcal{B}_n(\pi)$ denote the set of all Boolean intervals in $W(\Sym_n)$ with minimal element $\pi$.
Define the map $\varphi_\pi: \Out^{-1}(\pi) \rightarrow \mathcal{B}_n(\pi)$ by
\[
    \varphi_\pi(\beta) = \left[\pi, \pi \prod_{i\in I} s_i\right]
\]
where $I \subseteq \Asc(\pi)$ of nonconsecutive integers is determined by $\beta=\delta_I(\pi^{-1})$. 
Namely, the set  $I$ consists of the repeated values in $\beta$, which is unique by Proposition \ref{prop:fiberascents}.
We begin by establishing that $\varphi_\pi$ is a bijection. 

The output $\varphi_\pi(\beta)$ is computed using the unique set $I$ associated with each $\beta$, and hence is unique. Furthermore, the output $[\pi, \pi \prod_{i\in I} s_i]\in \mathcal{B}_n$ is a Boolean interval \cite[Corollary 4.4]{tenner2022interval}.
Therefore $\varphi_\pi$ is well-defined.

For injectivity: If $\varphi_\pi(\beta)=\varphi_\pi(\gamma)=[\pi, \pi \prod_{i\in I} s_i]$ for some (nonconsecutive) $I\subseteq \Asc(\pi)$, then  $\delta_{I}(\pi^{-1})=\beta$
and $\delta_{I}(\pi^{-1})=\gamma$. Therefore, $\beta=\gamma$.

For surjectivity: Every Boolean interval in $\mathcal{B}_n(\pi)$ has the form $[\pi, \pi \prod_{i\in I} s_i]$ where $I\subseteq \Asc(\pi)$ consists of nonconsecutive integers \cite[Corollary 4.4]{tenner2022interval}. Then, by Proposition~\ref{prop:fiberascents}, we know that $\delta_{I}(\pi^{-1})\in\Out^{-1}(\pi)$. Then $\varphi_\pi(\delta_I(\pi^{-1}))=[\pi, \pi \prod_{i\in I} s_i]$.

Together, this establishes that the map $\varphi_\pi$ is a bijection. 

Now define 
$\phi:\ufr{n}\to\mathcal{B}_n$ by
$\phi(\alpha)\coloneqq\varphi_\pi(\alpha)$
where $\Out(\alpha)=\pi$.
Since $\varphi_\pi$ is a bijection for all $\pi$ and since $\Out$ is well-defined (Remark \ref{rem:OutWellDef}), then $\phi$ is a bijection.

To conclude, we establish that $\varphi_\pi$ preserves the statistic of $n-k$ distinct ranks in $\Out^{-1}(\pi)$ and rank $k$ in the Boolean interval. 
Let $\beta \in \ufr{n}$ such that $\Out(\beta)=\pi$ where ties occur at ranks denoted by $r_1,r_2,\ldots,r_k$. Then $\beta$ has $n-k$ distinct ranks.
Then, by Proposition \ref{prop:fiberascents}, the set $I=\{r_1,r_2,\ldots ,r_{k}\}$ is a subset of $\Asc(\pi)$ consisting of $k$ nonconsecutive integers, and $\delta_I(\pi^{-1})=\beta$. 
Then $\varphi_\pi(\beta)$ corresponds uniquely to the rank $k$ Boolean interval given by $[\pi, \pi \prod_{i\in I} s_i]$.
\end{proof}

\section{Enumerations}
\label{sec: enumerations}

In this section, we provide enumerative formulas for:
\begin{enumerate}
    \item $f(n)$, the total Boolean intervals in $W(\Sym_n)$, \label{enum1}
    \item $f(n,k)$, the total number  of rank $k$ Boolean intervals in $W(\Sym_n)$,\label{enum2} and 
    \item the number of Boolean intervals in $W(\Sym_n)$ with minimal element $\pi$ \label{enum3}.
\end{enumerate}

To establish \eqref{enum1}, we begin with an immediate consequence of Theorem~\ref{thm:nkThm}.

\begin{corollary}\label{cor:ofmain}
    The total number of Boolean intervals in $W(\Sym_n)$ is equal to the number of unit Fubini rankings of length $n$.
\end{corollary}

By setting $q=1$ into the exponential generating function \cite[Exercise~3.185(h)]{stanley2012enumerative}
\begin{equation}
    F(x,q) 
    = \sum_{n\geq 0} \sum_{k\geq 0} f(n,k)q^k\frac{x^n}{n!} 
    = \frac{1}{1-x-\frac{q}{2}x^2}
    \label{eq:gen_fun},
\end{equation}
Stanley~\cite{stanley2016mathoverflow} points out that the total number of Boolean intervals in $W(\Sym_n)$ (OEIS \href{https://oeis.org/A080599}{A080599}) satisfies the recurrence relation  
\begin{equation}
    \label{eq: rec}
    f(n+1) = (n+1)f(n)+\binom{n+1}{2}f(n-1),
\end{equation}
where $f(0)=1$ and $f(1)=1$. 
In light of Corollary \ref{cor:ofmain}, we give a combinatorial proof of this result from the perspective of unit Fubini rankings.

\begin{restatable}{rethm}{recThm}
\label{thm: rec}
Let $g(n+1)$ denote the  number of unit Fubini rankings of length $n+1$. Then $g(n+1)$ satisfies the recursion
\[g(n+1)=(n+1)g(n)+\binom{n+1}{2}g(n-1),\]
where $g(1)=1$ and $g(2)=3$.
\end{restatable}

\begin{proof}
    Let $\alpha$ be unit Fubini ranking of length $n$. 
    The block structure of an element in $\ufr{n}$ means we have two options for the final block: it either ends in an $(n-1)(n-1)$ or an $n$. We have total freedom in the remaining positions. Thus there are two mutually exclusive cases to  consider.
    \begin{itemize}
        \item The last block has the form  $(n-1)(n-1)$: Then we may select one of the $g(n-1)$ unit Fubini rankings in $\ufr{n-1}$. Place the elements in the unit Fubini rankings in any of the $n+1$ possible spots for the unit Fubini ranking of length $n+1$. For each unit Fubini ranking in $\ufr{n-1}$ there are 
        \[
        \binom{n+1}{n-1} = \binom{n+1}{2}
        \]
        ways to do this.
        \item The last block has the form $n$: Then we may select one of the $g(n)$ unit Fubini rankings in $\ufr{n}$. Place the elements in the unit Fubini ranking in any of the $n+1$ possible spots for the unit Fubini ranking of length $n+1$. For each unit Fubini ranking in $\ufr{n}$ there are \[\binom{n+1}{n} = n+1\] ways to do this.
    \end{itemize}
    The recursion follows from taking the sum of the counts in each case. The initial values arise from the fact that $|\ufr{1}|=|\{(1)\}|$, hence $g(1)=1$, and $|\ufr{2}|=|\{(1,1),(1,2),(2,1)\}|$, hence $g(2)=3$.    
\end{proof}

For \eqref{enum2}, we begin with the following combinatorial proof.

\newFormulaThm*

\begin{proof}
Let $g(n,k)$ denote the number of unit Fubini rankings of length $n$ that have $n-k$ distinct ranks. 
Theorem \ref{thm:nkThm} implies that $g(n,k)=f(n,k)$, hence it suffices to show that $g(n,k)=\frac{n!}{2^k}\binom{n-k}{k}$.

If $\alpha\in\ufr{n}$ has $n-k$ distinct ranks, then its block structure has the form
\[b_1\,|\,b_2\,|\,\cdots\,|\,b_{n-k},\]
where exactly $k$ of the blocks have size two and all remaining blocks have size one. 
To enumerate all such $\alpha$, first select the indices of the blocks with size two. 
We can do this in $\binom{n-k}{k}$ ways. 
Then select the indices at which to place the repeated values within the blocks of size two. We do so iteratively, by first selecting two indices among $n$ where we will place the smallest repeated values of $\alpha$. 
This can be done in $\binom{n}{2}$ ways.
Then we repeat this process by selecting two indices among the remaining $n-2$ indices in which we place the next smallest repeated values of $\alpha$. 
This can be done in $\binom{n-2}{2}$ ways.
Through this process, the total ways in which we can place all repeated values in $\alpha$ is given by the product 
\[\binom{n}{2}\binom{n-2}{2}\cdots \binom{n-2(k-1)}{2}=\prod_{i=0}^{k-1}\binom{n-2i}{2}.\]
Finally, the values in the blocks of size one can appear in any order within the remaining available indices. 
We can place them in $(n-2k)!$ ways. Thus
\[
g(n,k)= \binom{n-k}{k}  (n-2k)! \prod_{i=0}^{k-1}\binom{n-2i}{2},
\]
    which simplifies to our desired result.
\end{proof}

\begin{remark}
In the introduction we referenced OEIS \href{https://oeis.org/A001286}{A001286}, a sequence known as the Lah numbers, which gives the values $f(n,1)=\frac{(n-1)n!}{2}$ for the number of $B_1$ in $W(\Sym_n)$. Theorem~\ref{newFormulaThm} implies that the Lah numbers also enumerate unit Fubini rankings with $n-1$ distinct ranks. 
Aguillon et al.~\cite{aguillon2023parking} showed that the number of unit interval parking functions in which exactly $n-1$ cars park in their preference is also enumerated by the Lah numbers. 
This result was established via a bijection between those parking functions and ideal states in the game the Tower of Hanoi, which were enumerated by the Lah numbers.
\end{remark}

We now prove that $g(n,k)$ has the same generating function as ~\eqref{eq:gen_fun}.
\begin{restatable}{rethm}{unnamedThm2}
    The exponential generating function for $g(n,k)$ has the closed form
    \[
    G(x,q) = \sum_{n\geq 0} \sum_{k\geq 0} g(n,k)q^k\frac{x^n}{n!} = \frac{1}{1-x-\frac{q}{2}x^2}.
    \]
\end{restatable}

\begin{proof}
    From Theorem~\ref{newFormulaThm}, we know that $g(n,k)= \frac{n!}{2^k}\binom{n-k}{k}$. Then
    \begin{align}\label{eq:gf1}
        G(x,q) = \sum_{n\geq 0} \sum_{k\geq 0} g(n,k)q^k\frac{x^n}{n!} =  \sum_{n\geq 0} \sum_{k\geq 0} \frac{1}{2^k}\binom{n-k}{k}q^kx^n.
    \end{align}
    Setting $n=0$ in Equation \eqref{eq:gf1} yields
    \begin{align} \label{eq:niszero}
         \sum_{k\geq 0} \frac{1}{2^k}\binom{-k}{k}q^kx^0= 1+ \sum_{k\geq 1} \frac{1}{2^k}\binom{-k}{k}q^k = 1 + 0.
    \end{align}
Substituting \eqref{eq:niszero} into \eqref{eq:gf1} gives
\begin{align}\label{eq:gen}
   G(x,q) &=  1+\sum_{n\geq 1} \sum_{k\geq 1} \frac{1}{2^k}\binom{n-k}{k}q^kx^n.
    \end{align}

Using the binomial identity $\binom{n}{k}=\binom{n-1}{k}+\binom{n-1}{k-1}$, \eqref{eq:gen} becomes

\begin{align}\label{eq:binomId}
        G(x,q) =  1+\sum_{n\geq 1} \sum_{k\geq 1} \frac{1}{2^k}\left ( \binom{n-k -1}{k} + \binom{n-k-1}{k-1}\right )q^kx^n,
    \end{align}
    which can be rewritten as
    \begin{align}\label{eq:foil}
            G(x,q) =  1+\sum_{n\geq 1} \sum_{k\geq 1} \frac{1}{2^k} \binom{n-k -1}{k}q^kx^n + \sum_{n\geq 1} \sum_{k\geq 1} \frac{1}{2^k} \binom{n-k-1}{k-1} q^kx^n.
    \end{align}
    The first set of summands in \eqref{eq:foil} simplifies in the following way:
    \begin{align} 
    \label{eq:reindex}
        \sum_{n\geq 1} \sum_{k\geq 1} \frac{1}{2^k} \binom{n-k -1}{k}q^kx^n 
        &= x\sum_{n\geq 1} \sum_{k\geq 1} \frac{1}{2^k} \binom{n- 1 - k}{k}q^kx^{n-1} \nonumber \\
        &= x\sum_{n\geq 0} \sum_{k\geq 0} \frac{1}{2^k} \binom{n -k}{k}q^kx^{n},
    \end{align}
    where the last equality in \eqref{eq:reindex} follows from re-indexing with respect to $n$, and the fact that $\binom{n}{k}=0$ whenever $k>n$.
The second set of summands in \eqref{eq:foil} simplifies in the following way:

    \begin{align} 
        \sum_{n\geq 1} \sum_{k\geq 1} \frac{1}{2^k} \binom{n-k-1}{k-1} q^kx^n 
        &= \frac{q}{2}x^{2} \sum_{n\geq 1} \sum_{k\geq 1} \frac{1}{2^{k-1}} \binom{n-2- k + 1)}{k-1} q^{k-1}x^{n-2} \nonumber\\
        &= \frac{q}{2}x^{2} \sum_{n\geq 0} \sum_{k\geq 0} \frac{1}{2^{k}} \binom{n-k}{k} q^{k}x^{n},\label{eq:reindex2}
    \end{align}
    where the last equality in \eqref{eq:reindex2} follows from re-indexing with respect to $n$ and $k$.

    Substituting \eqref{eq:reindex} and \eqref{eq:reindex2} into \eqref{eq:foil} allows us to reassemble everything to arrive at
    \[
    G(x,q) = 1+xG(x,q) + \frac{q}{2}x^{2}G(x,q),
    \]
    from which we arrive at
    \[
    G(x,q) = \frac{1}{1-x-\frac{q}{2}x^2}.\qedhere
    \]
\end{proof}

We now present our final enumerative result settling  \eqref{enum3}, which further connects this work to Fibonacci numbers.

\sizeOfFiberThm*

\begin{proof}
It is straightforward to prove that the number of ways to select nonconsecutive entries from the set $[n]$ is given by $F_{n+2}$. 
Thus, for each $i\in[k]$, the number of ways to select nonconsecutive elements from $b_i$ is given by $F_{|b_i|+2}$. As the blocks $b_1,b_2,\ldots,b_k$ are pairwise disjoint, the total number of ways to select subsets from $\cup_{i=1}^k b_i$ consisting of nonconsecutive integers is given by $\prod_{i=1}^k F_{|b_i|+2}$, as desired.
\end{proof}

Among the many results established by Tenner concerning Boolean intervals in both the Bruhat order and in the weak order \cite{tenner2022interval}, we highlight the following.

\begin{proposition}{\cite[Proposition 5.9]{tenner2022interval}}
     Let $i\in [n-1]$ be fixed. The number of Boolean intervals in $W(\Sym_n)$ of the form $[s_i, w]$ is $F_{i+1}F_{n-i+1}$, where $F_i$ is the $i$th Fibonacci number.
    \end{proposition}

For any $i\in[n-1]$, we have that $\Asc(s_i)=[n]\setminus\{i\}$. 
Then $b_1=[i-1]$ and $b_2=\{i+1,i+2,\ldots, n-1\}$, and Theorem \ref{thm: size of fiber} implies that the number of Boolean intervals in $W(\Sym_n)$ with minimal element $s_i$ is given by $F_{|b_1|+2}=F_{|b_2|+2}=F_{i+1}F_{n-i+1}$, recovering~\cite[Proposition 5.9]{tenner2022interval}.

We remark that in  the statement of Theorem \ref{thm: size of fiber}, we allow $[\pi,\pi]$ to be a Boolean interval. If we impose the condition that the maximal element $w$ cannot be equal to the minimal element $\pi$, then we have the following. 

\begin{corollary}
    Let $\pi = \pi_1 \pi_2 \cdots \pi_n \in \Sym_{n}$ be in one-line notation and partition its ascent set $\Asc(\pi)=\{i \in [n-1]: \pi_i<\pi_{i+1}\}$ into maximal blocks $b_1, b_2, \ldots, b_k$ of consecutive entries. 
Then, the number of Boolean intervals $[\pi,w]$ in $W(\Sym_n)$ with $w\neq \pi$ is given by
\begin{align*}
    \left(\prod_{i=1}^kF_{|b_i|+2}\right) -1,
\end{align*}
where $F_{\ell}$ is the $\ell$th Fibonacci number and $F_1=F_2 = 1$.

\end{corollary}
\begin{proof}
    The result follows from Theorem \ref{thm: size of fiber}, and noting that in creating a subset of $\Asc(\pi)$ consisting of nonconsecutive elements we cannot utilize the empty set.
\end{proof}

\section{Future work}\label{sec:future}

To the best of our knowledge the formula of Theorem \ref{newFormulaThm}, which counts the Boolean intervals of rank $k$ in $W(\Sym_n)$ did not exist previously in the literature. 
We gave a combinatorial proof of this result via the enumeration of unit Fubini rankings with $n-k$ distinct ranks. 
We wonder whether this new proof and combinatorial object might shed light on how a symmetric group proof may be constructed.

Tenner has provided many results for intervals in the weak (Bruhat) order \cite{tenner2022interval}. 
The paper also provides results on the Bruhat order, which leads us to wonder if there are other connections from Fubini rankings that can be used to count intervals in the Bruhat order.
We also wonder if it may be possible to utilize unit Fubini rankings, or a slight generalization thereof, to enumerate Boolean intervals in Bruhat and weak orders of other Coxeter systems. To this end we state the following: 
How many Boolean intervals are there in the weak order of the hyperoctahedral group?

\section*{Acknowledgements}
The authors thank Bridget E.~Tenner for helpful discussions and references. 
The authors thank Richard P.~Stanley for suggestions on a draft of this manuscript.
J.~Elder was partially supported through an AWM Mentoring Travel Grant. 
P.~E.~Harris was supported through a Karen Uhlenbeck EDGE Fellowship.
J.~C.~Mart\'inez Mori was partially supported by NSF Grant No. 2144127, awarded to S. Samaranayake. 
J.~C.~Mart\'inez Mori is supported by Schmidt Science Fellows, in partnership with the Rhodes Trust.

\bibliographystyle{plain}
\bibliography{bib_v2}

\begin{thebibliography}{10}

\bibitem{aguillon2023parking}
Yasmin Aguillon, Dylan Alvarenga, Pamela~E. Harris, Surya Kotapati, J.~Carlos Mart\'{\i}nez~Mori, Casandra~D. Monroe, Zia Saylor, Camelle Tieu, and Dwight~Anderson Williams, II.
\newblock On parking functions and the tower of {H}anoi.
\newblock {\em Amer. Math. Monthly}, 130(7):618--624, 2023.

\bibitem{bjorner2009shape}
Anders Bj\"{o}rner and Torsten Ekedahl.
\newblock On the shape of {B}ruhat intervals.
\newblock {\em Ann. of Math. (2)}, 170(2):799--817, 2009.

\bibitem{bradt2024unit}
S.~Alex Bradt, Jennifer Elder, Pamela~E. Harris, Gordon~Rojas Kirby, Eva Reutercrona, Yuxuan Wang, and Juliet Whidden.
\newblock Unit interval parking functions and the r-{F}ubini numbers.
\newblock {\em Matematica}, 3(1):370--384, 2024.

\bibitem{cayley1875analytical}
Arthur Cayley.
\newblock Xxviii. on the theory of the analytical forms called trees.
\newblock {\em The London, Edinburgh, and Dublin Philosophical Magazine and Journal of Science}, 13(85):172--176, 1857.

\bibitem{hadaway2022combinatorial}
K.~P. Hadaway.
\newblock On combinatorical problems of generalized parking functions.
\newblock Williams College, Honors Thesis, 2022.

\bibitem{harris2023outcome}
Pamela~E. Harris, Brian~M. Kamau, J.~Carlos Mart\'{\i}nez~Mori, and Roger Tian.
\newblock On the outcome map of {MVP} parking functions: permutations avoiding 321 and 3412, and {M}otzkin paths.
\newblock {\em Enumer. Comb. Appl.}, 3(2):Paper No. S2R11, 16, 2023.

\bibitem{konheim1966occupancy}
Alan~G. Konheim and Benjamin Weiss.
\newblock An occupancy discipline and applications.
\newblock {\em SIAM J. Appl. Math.}, 14(6):1266--1274, 1966.

\bibitem{mendelson1982races}
Elliott Mendelson.
\newblock Races with ties.
\newblock {\em Math. Mag.}, 55(3):170--175, 1982.

\bibitem{stanley2016mathoverflow}
R.~P. Stanley.
\newblock \url{https://mathoverflow.net/questions/253773/number-of-boolean-algebra-subintervals-in-weak-order-of-s-n/253781#253781}.
\newblock Online, accessed 06-12-2023.

\bibitem{stanley2012enumerative}
Richard~P. Stanley.
\newblock {\em Enumerative combinatorics. {V}olume 1}, volume~49 of {\em Cambridge Studies in Advanced Mathematics}.
\newblock Cambridge University Press, Cambridge, second edition, 2012.

\bibitem{tenner2006reduced}
Bridget~Eileen Tenner.
\newblock Reduced decompositions and permutation patterns.
\newblock {\em J. Algebraic Combin.}, 24(3):263--284, 2006.

\bibitem{tenner2007pattern}
Bridget~Eileen Tenner.
\newblock Pattern avoidance and the {B}ruhat order.
\newblock {\em J. Combin. Theory Ser. A}, 114(5):888--905, 2007.

\bibitem{tenner2022interval}
Bridget~Eileen Tenner.
\newblock Interval structures in the {B}ruhat and weak orders.
\newblock {\em J. Comb.}, 13(1):135--165, 2022.

\end{thebibliography}

\end{document}